\documentclass [a4paper,12pt]{amsart}
\usepackage{graphicx}
\usepackage{psfrag}
\usepackage[ansinew]{inputenc}    
\usepackage[all]{xy}

\usepackage{amssymb,amscd}

\newtheorem{thm}{Theorem}[section]
\newtheorem{exam}[thm]{Example}

\def\C{\mathbb C}

\def\dim{\operatorname{dim}}

\def\rank{\operatorname{rank}}
\def\codim{\operatorname{codim}}

\usepackage{graphicx}
\usepackage[ansinew]{inputenc}    
\usepackage[all]{xy}
\usepackage{amssymb,amscd}
\usepackage{color}

\newtheorem{cor}[thm]{Corollary}

\newtheorem{lem}[thm]{Lemma}

\theoremstyle{definition}

\newtheorem{rem}[thm]{Remark}

\def\C{\mathbb C}

\def\Q{\mathbb Q}

\def\dim{\operatorname{dim}}

\def\rank{\operatorname{rank}}
\def\codim{\operatorname{codim}}

\begin{document}
\title{Equisingularity of families of isolated determinantal singularities}

\author{J. J. Nuño-Ballesteros, B. Oréfice-Okamoto, J. N. Tomazella}

\address{Departament de Geometria i Topologia,
Universitat de Val\`encia, Campus de Burjassot, 46100 Burjassot
SPAIN}

\email{Juan.Nuno@uv.es}

\address{Departamento de Matem\'atica, Universidade Federal de S\~ao Carlos, Caixa Postal 676,
13560-905, S\~ao Carlos, SP, BRAZIL}

\email{bruna@dm.ufscar.br}

\address{Departamento de Matem\'atica, Universidade Federal de S\~ao Carlos, Caixa Postal 676,
13560-905, S\~ao Carlos, SP, BRAZIL}

\email{tomazella@dm.ufscar.br}

\thanks{The first author was partially supported by DGICYT Grant MTM2012--33073 and FAPESP Grant 2012/18805-2. 
The second author was partially supported by FAPESP Grant 2013/14014-3.
The third  author was partially supported by FAPESP Grant 2013/10856-0.
}

\subjclass[2000]{Primary 32S15; Secondary 58K60, 32S30} 

\keywords{Determinantal singularity, topological triviality, Whitney equisingularity}

\begin{abstract}
We study the topological triviality and the Whitney equisingularity of a family of isolated determinantal singularities. On one hand, we give a Lê-Ramanujam type theorem for this kind of singularities by using the vanishing Euler characteristic. On the other hand, we extend the results of Teissier and Gaffney about the Whitney equisingularity of hypersurfaces and complete intersections, respectively, in terms of the constancy of the polar multiplicities.
\end{abstract}

\maketitle
\section{Introduction}
The celebrated Lê-Ramanujan theorem \cite{LR} states that any 1-parameter family of isolated hypersurface singularities $\{(X_t,0)\}$ of dimension  $d\ne2$ has constant topological type if the Milnor number $\mu(X_t,0)$ is constant on $t$. The same result is also known to be true for isolated complete intersection singularities (ICIS) of dimension $d\ne2$ (see \cite{Param}). The condition $d\ne2$ is needed because of the use of the $h$-cobordism theorem in the proof. It is worth to mention that the case $d=2$ remains open until now. For space curves (not necessarily ICIS), the topological triviality of the family was obtained by Buchweitz-Greuel, by using an appropriate version of Milnor number (see \cite{BG}).

For plane curves, Zariski proved \cite{Z} that the $\mu$-constant condition is also equivalent to the fact that the family is Whitney equisingular. However, the Briançon-Speder example \cite{BS} shows that this is not true anymore for hypersurfaces of dimension $d\ge2$. Teissier solved completely in \cite{T} the  Whitney equisingularity question for hypersurfaces. He showed that the family is Whitney equisingular if and only if the $\mu^*$ sequence is constant on the family. This result was generalized by Gaffney \cite{G,G2} for ICIS of any dimension $d$, with the use of the polar multiplicities $m_i(X_0,0)$, $i=0,\dots,d$. Note that the polar multiplicities $m_i(X_0,0)$, for $i=0,\dots,d-1$ are defined for any variety $(X_0,0)$ (see \cite{LT}), but $m_d(X_0,0)$ is defined only for ICIS in \cite{G}. For space curves, a family is Whitney equisingular if and only if $m_1(X_t,0)$ is constant on $t$, where $m_1(X_t,0)$ is defined as the Milnor number of a generic linear projection (see \cite{NT}).

Here we study the topological triviality and the Whitney equisingularity, for families of isolated determinantal singularities (IDS). Determinantal varieties are those defined by the zeros of an ideal generated by the minors of a given size of a matrix. We assume that the variety has the expected dimension according to the sizes of the matrix and of the minors. Moreover, we restrict ourselves to the case that the determinantal variety has isolated singularity (in some sense related to the determinantal structure). Any ICIS (and hence any hypersurface) is a particular case of an IDS.

On one hand, we show that a family $\{(X_t,0)\}$ of IDS of dimension $d\ne 2$ has constant topological type provided that:
\begin{enumerate}
\item  the family is good (i.e., it has isolated singularity uniformly in the family), 
\item the determinantal smoothing of any member $(X_t,0)$ has the homotopy type of a bouquet of $d$-spheres,
\item the vanishing Euler characteristic $\nu(X_t,0)$ is constant on $t$.
\end{enumerate}
The number $\nu(X_0,0)$ was introduced in a previous paper \cite{NOT} for an IDS
$(X_0,0)$ and it is defined as  
$$
\nu(X_0,0)=(-1)^{d}(\chi(X_A)-1),
$$
where $X_A$ is the determinantal smoothing of $(X_0,0)$ and $\chi(\cdot)$ denotes the Euler characteristic. In the particular case of an ICIS, the determinantal smoothing coincides with the Milnor fibre and hence, $\nu(X_0,0)=\mu(X_0,0)$. Moreover, for ICIS the condition (2) is always satisfied and the condition (1) follows from (3). Thus, our result generalizes the ICIS version of the Lê-Ramanujam theorem.

On the other hand, we also show that a good family $\{(X_t,0)\}$ of IDS of any dimension $d$ is Whitney equisingular if and only if all the polar multiplicities $m_i(X_t,0)$, $i=0,\dots,d$ are constant on the family. Again, we find in \cite{NOT}  an adapted version of the top polar multiplicity $m_d(X_0,0)$ of an IDS $(X_0,0)$, which generalizes that of \cite{G} for ICIS. Hence, our theorem generalizes the results of Teissier or Gaffney in the particular cases of hypersurfaces or ICIS respectively.

\section{Isolated Determinantal Singularities}

We recall here from \cite{NOT} the main definitions and properties of IDS.
Let $0<s\leq m\leq n$ be integer numbers, $M_{m,n}$ be the set of complex matrices of size $m\times n$ and let $M_{m,n}^s$ be the subset of matrices with rank less than $s$.

We consider a holomorphic map germ $$f:(\C^N,0)\rightarrow M_{m,n}$$ and we say that the germ $$(X_0,0)=f^{-1}(M_{m,n}^s)$$is a determinantal germ of type $(m,n;s)$ if the dimension of $(X_0,0)$ is equal to $$d=N-(m-s+1)(n-s+1).$$

In \cite{NOT}, we introduced a new concept of isolated determinantal singularity. We defined the germ $(X_0,0)$ to be an isolated determinantal singularity, shortening IDS, if $s=1$ or $N<(m-s+2)(n-s+2)$, $X_0$ is smooth at $x$ and $\mbox{rank}(f(x))=s-1$, for all $x\neq 0$ in a neighborhood of the origin.

We fix a small enough representative $f:B_\epsilon\to M_{m,n}$, where $B_\epsilon$ is the open ball centered at the origin and of radius $\epsilon>0$. Given a matrix $A$ in $M_{m,n}$, let  $f_A:B_\epsilon\rightarrow M_{m,n}$ be the map defined by $$f_A(x)=f(x)+A$$ and let $$X_A=f_A^{-1}(M_{m,n}^s).$$ We see in \cite{NOT} that there exists a non-empty Zariski open set $W$ in $M_{m,n}$ such that $X_A$ is smooth if $A$ belongs to $W$ and the Euler characteristic of $X_A$, $\chi(X_A)$ does not depend on $A$ in $W$. The vanishing Euler characteristic of the IDS $(X_0,0)$ was, then, defined by
$$
\nu(X_0,0)=(-1)^{d}(\chi(X_A)-1),
$$
with $A$ in the Zariski open set $W$. We also call $X_A$ the special determinantal smoothing of $(X_0,0)$. In the case of an ICIS ($s=1$), $X_A$ can be seen as the Milnor fibre of $(X_0,0)$ and hence, $\nu(X_0,0)$ coincides with the Milnor number $\mu(X_0,0)$.

We recall now the notion of determinantal deformation. We consider a
 map germ $F:(\C^N\times \C,0)\rightarrow M_{m,n}$ such that $F(x,0)=f(x)$ for all $x\in \C^N$. Then we denote $\mathcal X=F^{-1}(M_{m,n}^s)$ and we say that the projection
 \begin{eqnarray*}
\pi:(\mathcal X,0)&\to&(\C,0)\\
(x,t)&\mapsto&t
\end{eqnarray*} 
is a determinantal deformation of $(X_0,0)$. If we fix again a representative $F:B_\epsilon\times D\to M_{m,n}$, then we set $f_t(x)=F(x,t)$ and $X_t=f_t^{-1}(M_{m,n}^s)$. 
 With this notation, it is common to write simply $X_t$ for the deformation, instead of $\pi$.

A determinantal smoothing of $(X_0,0)$ is a determinantal deformation $X_t$ such that $X_t$ is smooth and $\rank(f_t(x))=s-1$ for all $x\in X_t$ and for all $t\in D-\{0\}$. In this case, we showed in \cite{NOT}:
\begin{equation*} 
\nu(X_0,0)=(-1)^{d}(\chi(X_t)-1).
\end{equation*}

If $X_t$ is a determinantal deformation of $(X_0,0)$ as above, we say that:
\begin{enumerate}
\item $X_t$ is origin preserving if $0\in S(X_t)$, for all $t$ in $D$, where $S(X_t)$ denotes the singular set of $X_t$. Then we say that $\{(X_t,0)\}_{t\in D}$ is a $1$-parameter family of IDS;
\item $X_t$ is good if there exists $\epsilon>0$ such that $S(X_t)=\{0\}$ on $B_\epsilon$, for all $t$ in $D$;
\item $X_t$ has constant topological type if for each $t\in D$ there exists $\epsilon_t > 0$ such that $(B_{\epsilon_t},X_t)$ is homeomorphic to $(B_\epsilon,X_0)$.
\item $X_t$ is Whitney equisingular if it is good and $\{\mathcal X-T, T\}$ satisfies the Whitney conditions, where $T=\{0\}\times D$.
\end{enumerate}

Finally, we also recall the definition of polar multiplicities. They were defined by Lê-Teissier \cite{LT} for any $d$-dimensional variety $(X_0,0)\subset(\C^N,0)$ and for each $i=0,\dots,d-1$.  We choose a generic linear projection $p:\C^N\to\C^{d-i+1}$, then the $i$-th polar multiplicity is equal to
$$
m_i(X_0,0)=m_0(\overline{S(p|_{X_0^0})},0),
$$
where $X_0^0$ is the smooth part of $X_0$ and $m_0(Z,0)$ denotes the usual multiplicity of any variety $(Z,0)$. 

In the case that $(X_0,0)$ is an IDS, we introduced in \cite{NOT} the $d$-th polar multiplicity as
$$
m_d(X_0,0)=\#S(p|_{X_A}),
$$
where now $X_A$ is a determinantal smoothing and $p:\C^N\to\C$ is a generic linear function.

\section{The topology of the determinantal smoothing}

In this section, we describe the topology of the determinantal smoothing $X_A$ of an IDS $(X_0,0)$ of dimension $d$. 
We know that $X_A$ has the homotopy type of a CW-complex of real dimension $d$. We start by showing that the polar multiplicities of $(X_0,0)$ are related to the cell structure of the smoothing $X_A$.

\begin{lem}\label{tipo kiumars} Let $M=X_A\cap B_{\epsilon'}$ for some $0<\epsilon'<\epsilon$. Let $g:X_A\rightarrow\C$ be a holomorphic Morse function with no critical points on $X_A\cap S_{\epsilon'}$. Then,
\begin{equation*}
M\cong(g^{-1}(c)\cap M) \, \mbox{with $m$ $d$-cells attached,} 
\end{equation*}
where $m=\#S(g|_M)$ and $c$ is a regular value of $g$.
\end{lem}
\begin{proof}
See the proof of Theorem A.5 of \cite{NOT}.
\end{proof}

\begin{thm}
Let $(X_0,0)$ be a $d$-dimensional IDS and let $X_A$ be a determinantal smoothing. 
Then $X_A$ has the homotopy type of CW-complex with $m_i(X_0,0)$ cells of dimension $i$, for $i=0,\dots,d$.
\end{thm}

\begin{proof}
Let's proceed by finite induction over $d$. If $d=0$, $X_A$ is a finite set with $\chi(X_A)$ points. By \cite[3.5]{NOT},
$$\chi(X_A)=\nu(X_0,0)+1=\dim_\C\mathcal O_{X_0,0}=m_0(X_0,0).$$

Assume that the result is true for any IDS with dimension $d-1$. We choose a generic linear map $p:\C^N\to\C$. Then by Lemma \ref{tipo kiumars},
$$
X_A\cong(p^{-1}(c)\cap X_A) \, \mbox{with $\#S(p|_{X_A})$ $d$-cells attached.}
$$
We have that $p^{-1}(c)\cap X_A$ is the smoothing of the $(d-1)$-dimensional IDS $(X_0\cap p^{-1}(0),0)$. By induction hypothesis, $p^{-1}(c)\cap X_A$ has a CW-complex structure with $m_i(X_0\cap p^{-1}(0),0)$ cells of dimension $i$, for $i=0,\dots,d-1$. But it follows from \cite{LT} that $m_i(X_0,0)=m_i(X_0\cap p^{-1}(0),0)$ for $i=0,\dots,d-1$, and $m_d(X_0,0)=\#S(p|_{X_A})$, by definition of the $d$-th polar multiplicity.  
\end{proof}

\begin{cor}
If $(X_0,0)$ is a $d$-dimensional IDS, then
\begin{equation*}
\nu(X_0,0)=(-1)^d(\sum_{i=0}^d(-1)^im_i(X_0,0)-1).
\end{equation*}
\end{cor}

\begin{proof}
By definition $\nu(X_0,0)=(-1)^d(\chi(X_A)-1)$, then
\begin{eqnarray*}
\nu(X_0,0)&=&(-1)^d(\chi(X_A)-1)\\
&=&(-1)^d(\sum_{i=0}^d(-1)^im_i(X_0,0)-1).
\end{eqnarray*}
\end{proof}

\begin{exam} {\rm Let $(X_0,0)$ be the determinantal surface in $(\C^4,0)$ given by the $2\times 2$ minors of the matrix:
\[
\left(
\begin{array}{ccc}
x  & y  & z  \\
y  & z  & w  
\end{array}
\right)
\]
We showed in \cite{NOT} that $m_0(X_0,0)=3$, $m_1(X_0,0)=4$ and $m_2(X_0,0)=3$. Thus, any determinantal smoothing $X_A$ has the homotopy type of a 2-dimensional CW-complex with 3 0-cells, 4 1-cells and 3 2-cells.}
\end{exam}

Another interesting topological property is that any determinantal smoothing is always connected (unless it is 0-dimensional). 

\begin{lem}\label{XAconexo}
If $(X_0,0)$ is an IDS of dimension $d\ge 1$, then $X_A$ is connected.
\end{lem}

\begin{proof}
From the long exact homotopy sequence of the pair $(\overline{X_A},\partial \overline{X_A})$, the canonical mapping
\begin{equation*}
\pi_i(\partial \overline{X_A})\to \pi_i(\overline{X_A})
\end{equation*}
is surjective if $i=d-1$ and is an isomorphism if $i\leq d-2$ (see \cite{St}, proof of Theorem 1). Moreover, we have $\pi_i(\partial \overline{X_0})=\pi_i(\partial \overline{X_A})$ and $\pi_i(\overline{X_A})=\pi_i(X_A)$.

If $d\geq 2$, $(X_0,0)$ is normal, therefore irreducible and $\partial \overline{X_0}$ is connected. Hence, $\pi_0(X_A)\cong\pi_0(\partial \overline{X_0})$ and $X_A$ is connected. If $d=1$, $X_A$ is connected by Theorem 4.2.2 of \cite{BG}.
\end{proof}

It was shown by Milnor  \cite{M}, that the Milnor fibre of a $d$-dimensional isolated hypersurface singularity has the homotopy type of a bouquet of spheres of dimension $d$. This was generalized later by Hamm \cite{Hamm} for ICIS. For the purpose of the next section, it would be interesting to know whether this is also true for general IDS.

When $d=1$, this is always true, since any connected 1-dimensional CW-complex has the homotopy type of a bouquet of 1-spheres. For $d=2$, we know from \cite{St} that $\beta_1(X_A)=0$, but this is not enough to ensure that $X_A$ has the homotopy type of a bouquet of 2-spheres, since we should need that $X_A$ is simply connected. We do not know at this moment if this is true or not for a 2-dimensional IDS. For $d\ge 3$, the result is definitely false in general: Damon and Pike \cite{DP} give examples of 3-dimensional IDS for which $\beta_2(X_A)\ne0$.

In general, we will use the following well known criterion for a $d$-dimensional CW-complex to be a bouquet of spheres of dimension $d$. The proof follows easily by using Hurewicz and Whitehead theorems (see also the proof of \cite[6.5]{M}).

\begin{lem}
Let $X$ be a connected $d$-dimensional CW-complex, with $d\ge2$. The following statements are equivalent:
\begin{enumerate}
  \item $H_2(X)=\dots=H_{d-1}(X)=\pi_1(X)=1$;
  \item $X$ is $(d-1)$-connected;
  \item $X$ has the homotopy type of a bouquet of spheres of dimension $d$.
\end{enumerate}
\end{lem}

Finally, we deduce the following direct consequence of the previous lemma and the results of \cite{St}.

\begin{cor}
If $(X_0,0)$ is a $d$-dimensional IDS in $\C^{d+k}$ with $k\leq d-1$ and such that $H_{d-k+1}(X_A)=\dots=H_{d-1}(X_A)=0$, then $X_A$ has the homotopy type of a bouquet of $\nu(X_0,0)$ spheres of dimension $d$.
\end{cor}

\section{The topology of a determinantal deformation}

We consider here a determinantal deformation $(\mathcal X,0)$ of the IDS $(X_0,0)$. If the deformation is given by $(\mathcal X,0)=F^{-1}(M_{m,n}^s)$, we fix a representative $F:B\times D\to M_{m,n}$, where $B\subset\C^N$ and $D\subset\C$ are small enough open balls centered at the origin. We set the following notation for each $t\in D$ and $A\in M_{m,n}$:
\begin{itemize}
\item $f_{t}:B\rightarrow M_{m,n}$ the map defined by $f_t(x)=F(x,t)$ and $X_t=f_t^{-1}(M_{m,n}^s)$;
\item $f_{t,A}:B\rightarrow M_{m,n}$ the map defined by $f_{t,A}(x)=f_t(x)+A$ and $X_{t,A}=f_{t,A}^{-1}(M_{m,n}^s)$.
\end{itemize}

\begin{lem}\label{smoothing} The subset $W\subset D\times M_{m,n}$ of pairs $(t,A)$ such that $X_{t,A}$ is smooth and  $\rank f_{t,A}(x)=s-1$ for any $x\in X_{t,A}$ is non-empty Zariski open. Moreover, the topological type of $X_{t,A}$ does not depend on $(t,A)$ in $W$.
\end{lem}

\begin{proof} Just follow the arguments of \cite[Theorem 3.4]{NOT}.
\end{proof}

%

We fix some notations which will be useful for some of the results of this section.
We choose the ball $B_\epsilon$ such that the only singular point of $X_0$ in $B_\epsilon$ is the origin. Now we fix a generic matrix $A\in M_{m,n}$ such that $X_A$ is a determinantal smoothing of $(X_0,0)$. Moreover, since the subset $W$ of Lemma \ref{smoothing} is open, we can shrink the disk $D$ if necessary, in such a way that for any $t\in D$, $(t,A)\in W$. Hence, $X_{t,A}$ is smooth, $\rank f_{t,A}(x)=s-1$ for any $x\in X_{t,A}$ and $X_{t,A}$ is homeomorphic to $X_A$. 

Fix $t\in D$ and let $\{x_1,\dots,x_{k}\}$ be the singular set of $X_t\cap B_\epsilon$. For each $i=1,\dots,k$, choose an open disc $D_i\subset B_\epsilon$ centered at $x_i$ and such that $x_i$ is the only critical point of $X_t\cap D_i$ (see fig. \ref{deform}). It follows that $X_{t,A}\cap D_i$ is a determinantal smoothing of $(X_t,x_i)$. Moreover, we can choose a continuous map $r:X_{t,A}\to X_t$ such that $r(X_{t,A}\cap D_i)=\{x_i\}$ and $r$ is a homeomorphism outside $D_i$.

Finally, we set:
\begin{eqnarray*}
U&=&X_t\cap(\cup_{i=1}^{k}\overline{D}_i),\\ 
V&=&X_t\cap(B_\epsilon-\cup_{i=1}^{k_t}D_i^t),\\
\tilde{U}&=&X_{t,A}\cap(\cup_{i=1}^{k}\overline{D}_i),\\ 
\tilde{V}&=&X_{t,A}\cap(B_\epsilon-\cup_{i=1}^{k}D_i).
\end{eqnarray*}

\begin{figure}[ht!]
\begin{center}
   \psfrag{B}{$B_\epsilon$}
  \psfrag{X0}{$X_0$}
\psfrag{Xt}{$X_t$}
 \psfrag{XtA}{$X_{t,A}$}
  \psfrag{x}{$0$}
 \psfrag{x1}{$x_1$}
\psfrag{x2}{$x_2$}
\psfrag{x3}{$x_3$}
 \psfrag{D1}{$D_1$}
\psfrag{D2}{$D_2$}
\psfrag{D3}{$D_3$}
\includegraphics[height=7cm]{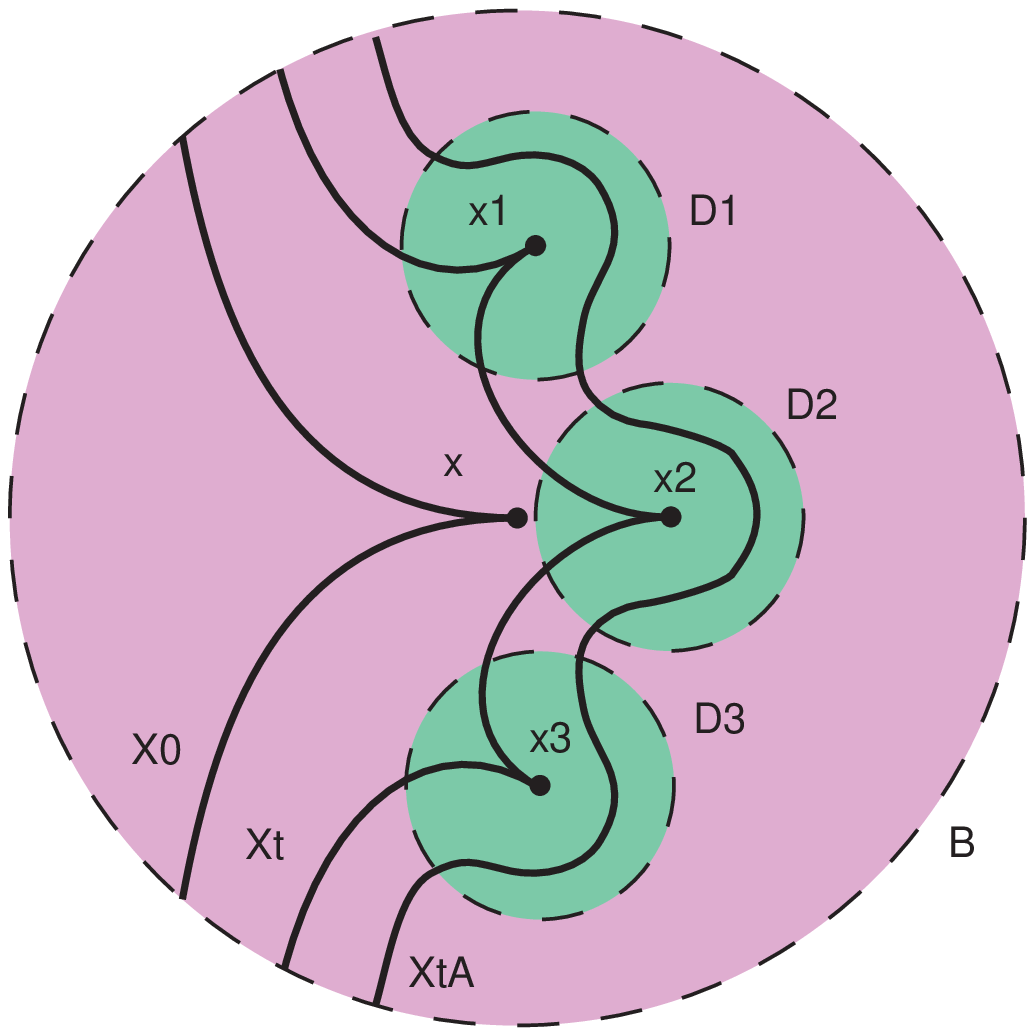}
\caption{}\label{deform}
\end{center}
\end{figure}

The following theorem states that the difference between the global invariants in the fibres of a deformation measures exactly the vanishing Euler characteristic of the deformation. This fact is well known for curve singularities (see \cite{BG}), for hypersurfaces (see \cite{Si}) and probably for ICIS, although we could not find any good reference in the literature. We prove it here for general IDS.

\begin{thm}\label{chiXt} Let $(X_0,0)$ be a $d$-dimensional IDS and let $X_t$ be any determinantal deformation. Then,
\begin{equation*}
\nu(X_0,0)-\sum_{x\in S(X_t)}\nu(X_t,x)=(-1)^d(\chi(X_t)-1)
\end{equation*} 
\end{thm}

\begin{proof}
With the above notation we have for $X_t$:
$$
\chi(X_t)=\chi(U\cup V)=\chi(U)+\chi(V)-\chi(U\cap V)=\chi(V)+k,
$$
because each $D_i\cap X_t$ is contractible and $\chi(U\cap V)=0$ (since $U\cap V$ is a closed manifold of odd dimension).
 
Analogously, for $X_{t,A}$ we get:
$$
\chi(X_t)=\chi(\tilde U\cup \tilde V)=\chi(\tilde U)+\chi(\tilde V)-\chi(\tilde U\cap \tilde V)=\chi(\tilde V)+\sum_{i=1}^{k} \chi(X_{t,A}\cap D_i).
$$

On the other hand,
\begin{eqnarray*}
(-1)^d\nu(X_0,0)+1&=&\chi(X_A)\\
&=&\chi(X_{t,A})\\
&=&\chi(\tilde V)+\sum_{i=1}^{k} \chi(X_{t,A}\cap D_i)\\
&=&\chi(V)+\sum_{i=1}^{k}((-1)^d\nu(X_t,x_i)+1)\\
&=&\chi(X_t)-k+\sum_{i=1}^{k}(-1)^d\nu(X_t,x_i)+k\\
&=&\chi(X_t)+(-1)^d\sum_{i=1}^{k}\nu(X_t,x_i),
\end{eqnarray*}
which gives the desired equality.
\end{proof}

\begin{cor}
Let $X_t$ be a determinantal deformation of an IDS $(X_0,0)$. Then, the sum $\sum_{x\in S(X_t)}\nu(X_t,x)$ is constant on $t$ if and only if $\chi(X_t)=1$.
\end{cor}

We recall that a deformation $(\mathcal X,0)$ of $(X_0,0)$ is topologically trivial if there exists a homeomorphism $\Phi:(\mathcal X,0)\to(X_0\times\C,0)$ which commutes with the projection, that is, $\pi\circ\Phi=\pi$, where $\pi(x,t)=t$. In particular, this implies that $X_t$ is homeomorphic to $X_0$, for $t$ small enough. 
 
\begin{cor}
If $X_t$ is a topologically trivial determinantal deformation of an IDS $(X_0,0)$, then
\begin{equation*}
\sum_{x\in S(X_t)}\nu(X_t,x)=\nu(X_0,0).
\end{equation*}
\end{cor}

\begin{rem}
We know from  \cite{M} that the Milnor number of a hypersurface with isolated singularity is a topological invariant, that is, if $(X_0,0)$ and $(X_0',0)$ are hypersurfaces with isolated singularity and have the same topological type (i.e., $(B_\epsilon,X_0)$ is homeomorphic to $(B_{\epsilon'},X_0')$), then $\mu(X_0,0)=\mu(X_0',0)$. However, this is not true in general for higher codimension. For instance, take $(X_0,0)\subset(\C^3,0)$ as the determinantal curve given by the $2\times 2$-minors of 
\[
\left(
\begin{array}{ccc}
x  & y  & z  \\
y  & z  & x^2  
\end{array}
\right)
\]
and take $(X_0',0)$ as the $x$-axis $y=z=0$. Both curves have the same topological type, since both are irreducible, $(X_0,0)$ is the curve parametrized by $t\mapsto(t^3,t^4,t^5)$ and $(X_0',0)$ is smooth. However, $\mu(X_0,0)=4$ and $\mu(X_0',0)=0$.
\end{rem}

In order to deduce some other consequences of Theorem \ref{chiXt}, we need some connection properties of the general fibre $X_t$ of a determinantal deformation of an IDS $(X_0,0)$. 

\begin{lem} \label{Xtconexo}
If $(X_0,0)$ is an IDS with dimension $d\ge1$ then any determinantal deformation $X_t$ is connected for $t$ small enough.
\end{lem}
\begin{proof}
We consider the Mayer-Vietoris sequences of the pairs $(U,V)$ in $X_{t}$ and $(\tilde U,\tilde V)$ in $X_{t,A}$ where the vertical arrows are induced by the continuous map $r:X_{t,A}\to X_t$: 
$$
\begin{CD}
\dots @>>> H_0(U\cap V) @>>> H_0({U})\oplus H_0({V})  @>>> H_{0}(X_t) @>>> 0\\
@.     @AA\alpha_0A   @AA\beta_0 A   @AA\gamma_0 A  @.\\
\dots @>>> H_0(\tilde{U}\cap \tilde{V}) @>>> H_0(\tilde{U})\oplus H_0(\tilde{V}) @>>> H_0(X_{t,A}) @>>> 0
\end{CD}
$$

Note that $\beta_0$ is an isomorphism, since $r$ is a homeomorphism between $\tilde V$ and $V$ and it is also a bijection between the connected components of $\tilde U$ and $U$. Then, $\gamma_0$ is an epimorphism by elementary calculatios. On the other hand, since $X_A\simeq X_{t,A}$ is connected by Lemma \ref{XAconexo}, we have that $X_t$ is also connected.
\end{proof}

\begin{lem}\label{beta1Xt}
If $(X_0,0)$ is an IDS with dimension $d\ge2$, then $\beta_1(X_t)=0$  for any determinantal deformation $X_t$ and for $t$ small enough.
\end{lem}

\begin{proof}
We consider the Mayer-Vietoris sequence of the pair $(U,V)$ in $X_t$ for the homology with rational coefficients:
\begin{align*}
&\begin{CD}
\dots@>>> H_1({U};\Q)\oplus H_1({V};\Q)  @>\phi_1 >>  H_{1}(X_t;\Q) @>\phi_2 >>
\end{CD}\\
&\begin{CD}
@>\phi_2 >> H_0(U\cap V;\Q) @>\phi_3 >>  H_0(U;\Q)\oplus H_0(V;\Q) @>>> \dots
\end{CD}
\end{align*}
Since $d\ge 2$, all the singularities $(X_t,x_i)$ are normal and hence, irreducible. Thus, $U\cap V$ has exactly $k$ connected components (one for each singular point) and the induced morphism $H_0(U\cap V;\Q)\to H_0(U;\Q)$ is an isomorphism. This implies that  $\phi_3$ is a monomorphism,  $\phi_2=0$ and hence, $\phi_1$ is an epimorphism. In fact, since $H_1({U};\Q)=0$, we can consider $\phi_1:H_1(V;\Q)\to H_1(X_t;\Q)$.

Thus, we have
$$
\begin{CD}
H_1({V};\Q)  @>\phi_1>> H_{1}(X_t;\Q)\\
@AA\beta_1 A   @AA\gamma_1 A  @.\\
H_1(\tilde{V};\Q) @>>> H_1(X_{t,A};\Q)
\end{CD}
$$
where $\phi_1$ is epimorphism and $\beta_1$ is isomorphism. Then, $\gamma_1$ is also epimorphism. But $(X_0,0)$ is normal and by \cite{St}, $H_1(X_{t,A};\Q)=H_1(X_A;\Q)=0$, therefore $H_1(X_t;\Q)=0$.
\end{proof}

As a consequence of the two previous lemmas and Theorem \ref{Xtconexo}, we can conclude that the vanishing Euler characteristic of $2$-dimensional IDS is upper semi-continuous.

\begin{cor} If $(X_0,0)$ is a $2$-dimensional IDS, then $\nu(X_t,x)\le \nu(X_0,0)$ for any determinantal deformation $X_t$ and for any $x\in S(X_t)$. 
\end{cor}

\begin{proof} We know that $\nu(X_t,x)=\beta_2(X_{t,A})\ge0$, for any $x\in S(X_t)$. Hence, by \ref{chiXt}, \ref{Xtconexo} and \ref{beta1Xt}:
$$
\nu(X_t,x)\le\sum_{x\in S(X_t)} \nu(X_t,x)=\nu(X_0,0)-\chi(X_t)+1=\nu(X_0,0)-\beta_2(X_t)\le\nu(X_0,0).
$$
\end{proof}

In order to proof the main theorem of this section (Theorem \ref{tipotopconst}), we will need to use the fact that $X_t$ is contractible. This could be concluded from Theorem \ref{Xtconexo} if we knew that $X_t$ was $(d-1)$-connected. In the two following lemmas, we show that this is true if the determinantal smoothing of $(X_t,x)$ is assumed to be $(d-1)$-connected, for any singular point $x$.

\begin{lem}\label{lema1}
Let $(X_0,0)$ be an IDS of dimension $d\ge 2$ and let $\mathcal X$ be a determinantal deformation of $(X_0,0)$. If there is a determinantal smoothing $X_A$ of $(X_0,0)$ such that $\pi_1(X_A)=1$, then $\pi_1(X_t)=1$ for $t$ small enough.
\end{lem} 

\begin{proof} In this proof we change slightly the notation for the sets $U,V,\tilde U,\tilde V$ from the beginning of this section. Since we are going to use Van Kampen theorem, we need that these sets are open.


Let $D_i'\subset D_i$ be another ball with center $x_i$. Let us denote 
$$
V=X_t-(\cup_{i=1}^nD_i'),\quad \tilde{V}=X_A-(\cup_{i=1}^nD_i'),
$$
 and, for each $i=1, \dots, k$, 
$$
U_i=X_t\cap D_i,\quad U=U_1\cup\dots\cup U_k.
$$

We observe that $U_1$ is contractible (because $D_1$ is a Milnor ball for $(X_t,x_1)$) and, therefore, connected and $U_1\cap V$ is also connected (because it has the same homotopy type of the link of $(X_t,x_1)$, see the proof of Lemma \ref{XAconexo}). We will show that $V$ is also connected. In fact, we consider the following Mayer-Vietoris sequence in the reduced homology:
$$\xymatrix{\dots\ar[r]&\tilde H_0(U\cap V)\ar[r]&\tilde H_0(U)\oplus\tilde H_0(V)\ar[r]&\tilde H_0(X_t)}.$$
Since $\tilde H_0(X_t)=0$ by Lemma \ref{Xtconexo}, the map
$$\xymatrix{\tilde H_0(U\cap V)\ar[r]&\tilde H_0(U)\oplus\tilde H_0(V)}$$
is surjective. However, the number of connected components of $U\cap V$ is equal to the number of connected components of $U$, so $\tilde H_0(U\cap V)\to\tilde H_0(U)$ is an isomorphism and  therefore, $\tilde H_0(V)=0$.

By the Van Kampen theorem, the induced map
$$\xymatrix{ \pi_1(U_1)*_{\pi_1(U_1\cap V)}\pi_1(V)\ar[r]&\pi_1(U_1\cup V)}$$
is an isomorphism. But $\pi_1(U_1)=1$, hence 
$$\xymatrix{\pi_1(V)\ar[r]&\pi_1(U_1\cup V)}$$ 
is surjective.

Now, we denote $V_2=V\cup U_1$. We know, similarly, that $U_2$, $V_2$ e $U_2\cap V_2$ are connected. By the Van Kampen theorem, the induced map
$$\xymatrix{ \pi_1(U_2)*_{\pi_1(U_2\cap V_2)}\pi_1(V)\ar[r]&\pi_1(U_2\cup V_2)}$$
is an isomorphism. But again $\pi_1(U_2)=1$, so $pi_1(V_2)\to\pi_1(U_1\cup V_2)$ is surjective. By taking composition with the above map, we have that
$$\xymatrix{\pi_1(V)\ar[r]&\pi_1(U_1\cup U_2\cup V)}$$
is also surjective.

Proceeding in this way, we conclude that there the following map is surjective:
$$\xymatrix{\pi_1(V)\ar[r]&\pi_1(U_1\cup\dots\cup U_n\cup V)=\pi_1(X_t)}.$$

We have, then, the following commutative diagram
$$\xymatrix{\pi_1(V)\ar[r]&\pi_1(X_t)\\ \pi_1(\tilde V)\ar[u]\ar[r]&\pi_1(X_A)=1\ar[u]}$$
where the top line is an epimorphism and the left column is an isomorphism. Therefore, $\pi_1(X_t)=1$.
\end{proof}

\begin{lem}\label{d-1conexo}
Let $X_0$ be an IDS of dimension $d\ge 1$ and let $X_t$ be a determinantal deformation of it such that the determinantal smoothing of $(X_t,x)$ is $(d-1)$-connected, for all $x\in S(X_t)$ and $t$ small enough. Then $X_t$ is also $(d-1)$-connected, for all $t$ small enough.
\end{lem}

\begin{proof}
If $d=1$, the result follows from Lemma \ref{Xtconexo}. If $d\ge 2$, from Lemma \ref{lema1}, we know that $\pi_1(X_t)=1$ for $t$ small enough. We will show that $H_i(X_t)=0$ for any $i=2, \dots, k$. 

We keep the notation of the beginning of this section for the sets $U,V,\tilde U,\tilde V$.
We consider again the Mayer-Vietoris sequences as in the proof of Lemma \ref{Xtconexo}:
\begin{align*}
&\begin{CD}
\dots @>>>  H_i(U\cap V) @>>>  H_i({U})\oplus  H_i({V})  @>>>  H_{i}(X_t) @>>> \\
@.     @AA\alpha_i A   @AA\beta_i A   @AA\gamma_i A \\
\dots @>>>  H_i( \tilde U\cap \tilde V) @>>>  H_i({\tilde U})\oplus H_i({\tilde V})  @>>>  H_{i}(X_{t,A}) @>>>
\end{CD}\\
&\begin{CD}
@>>>  H_{i-1}(U\cap V) @>>>  H_{i-1}({U})\oplus  H_{i-1}({V})  @>>>  H_{i-1}(X_t) @>>>\dots\\
@. @AA\alpha_{i-1}A   @AA\beta_{i-1} A   @AA\gamma_{i-1} A @.\\
@>>> H_{i-1}(\tilde U\cap \tilde V) @>>>  H_{i-1}({\tilde U})\oplus  H_{i-1}({\tilde V})  @>>>  H_{i-1}(X_{t,A}) @>>>\dots
\end{CD}
\end{align*}

We use an appropriate version of the five lemma (see \cite{H}). We have: 
\begin{enumerate}
\item $\beta_i$ is an epimorphism, since $r$ is a homeomorphism between $\tilde V$ and $V$ and $H_i(U)=0$ (each connected component $U\cap D_j$ is contractible). 
\item $\alpha_{i-1}$ is an isomorphism since $r$ is a homeomorphism between $\tilde U\cap \tilde V$ and $U\cap V$. 
\item $\beta_{i-1}$ is a monomorphism, since $H_{i-1}({\tilde U})=0$ from the hypothesis and $\beta_{i-1}$ is an isomorphism between $ H_{i-1}({\tilde V})$ and $H_{i-1}({V})$. 
\end{enumerate}
Therefore, $\gamma_i$ is an epimorphism and since $H_i(X_{t,A})=0$, we also have $ H_i(X_{t})=0$.

\end{proof}

It is known that any $1$-parameter family of hypersurfaces, or more generally ICIS, with dimension different of two has constant topological type if and only if it has constant Milnor number (see \cite{LR, Param}). In the following theorem, which is the main result of this section, we generalize this fact for a family of IDS. Our proof follows the arguments given in \cite{BG, LR, Param} for space curves, hypersurfaces or ICIS, respectively.

\begin{thm}\label{tipotopconst}
Let $\{(X_t,0)\}$ be a good family of IDS of dimension $d\neq 2$ such that the determinantal smoothing of $(X_t,0)$ is $(d-1)$-connected and $\nu(X_t,0)$ is constant for $t$ small enough. Then the family has constant topological type.
\end{thm}

\begin{proof}

We observe that we have three cases to consider: 

Case 1: If $(X,0)$ is a curve germ. In this case the theorem is showed in \cite{BG}.

Case 2: If $(X,0)$ is an hypersurface germ. This result is showed in \cite{LR}.

Case 3: If $d>2$ and $\mbox{codim} X_0>1$. We will proceed to the proof of this case.

Choose an $\epsilon >0$ small enough such that the origin is the only critical point in $X_0\cap B_\epsilon$ and such that $X_0$ is transversal to $S_\eta$ for any $\eta$ such that $0<\eta\leq\epsilon$. If $t$ is small enough, $X_t$ is also transversal to $S_\epsilon$. Consider the map
\begin{eqnarray*}
\pi:\mathcal X\cap(S_\epsilon\times \C)&\to& D\subset\C\\
(x,t)&\mapsto&t
\end{eqnarray*}
Since $\pi$ is a proper submersion, by the Ehresmann fibration theorem, $\pi$ is a locally trivial fibration. Therefore for $t$ small enough $(S_\epsilon, X_t\cap S_\epsilon)$ is diffeomorphic to $(S_\epsilon,X_0\cap S_\epsilon)$.

We know that $(B_\epsilon,X_0\cap B_\epsilon)$ is homeomorphic to $(C(S_\epsilon), C(X_0\cap S_\epsilon))$, where $C(M)$ denotes the cone of $M$.
For each $t$, let $0<\epsilon_t<\epsilon$ such that $0$ is the only critical point of $X_t\cap B_{\epsilon_t}$ and $X_t$ is transversal to $S_\eta$ for any $0<\eta\leq\epsilon_t$. Then $(B_{\epsilon_t},X_t\cap\epsilon_t)$ is homeomorphic to $(C(S_{\epsilon_t}),C(X_t\cap S_{\epsilon_t}))$. 
We will show that $(S_{\epsilon_t},X_t\cap S_{\epsilon_t})$ is homeomorphic to $(S_\epsilon,X_t\cap S_\epsilon)$ and, therefore,

\begin{equation*}
(B_\epsilon,X_0)\cong(C(S_\epsilon),C(X_0\cap S_\epsilon))
\cong(C(S_{\epsilon}),C(X_t\cap S_{\epsilon}))\cong(C(S_{\epsilon_t}),C(X_t\cap S_{\epsilon_t}))\cong (B_{\epsilon_t},X_t).
\end{equation*}

Let $M= B_\epsilon-\stackrel{\circ}{B_{\epsilon_t}}$ and $V=M\cap X_t$. The boundary of $M$ is $\partial M=S_\epsilon\sqcup S_{\epsilon_t}$ and the boundary of $V$ is $\partial V=(X_t\cap S_\epsilon)\sqcup (X_t\cap S_{\epsilon_t})$ (see fig. \ref{cobordismo}). Obviously, the inclusions $S_\epsilon,S_{\epsilon_t} \hookrightarrow M$ are homotopy equivalences, we show that $X_t\cap S_\epsilon,X_t\cap S_{\epsilon_t}\hookrightarrow V$
 are also homotopy equivalences. By Whitehead and Hurewicz theorems, we only need to prove that the induced maps between the fundamental groups and the homology groups are isomorphisms. 
  
We have that $X_t\cap B_{\epsilon_t}$ and $X_t\cap B_\epsilon$ are contractible (in the first one $B_{\epsilon_t}$ is a Milnor ball for $(X_t,0)$ and the second one follows from Lemma \ref{d-1conexo} and Theorem \ref{chiXt}). By Van Kampen theorem, 
$$
1=\pi_1(X_t\cap B_\epsilon)=\pi_1(X_t\cap B_{\epsilon_t})*_{\pi_1(X_t\cap S_{\epsilon_t})} \pi_1(V),
$$
but $\pi_1(X_t\cap B_{\epsilon_t})=1$, $\pi_1(X_t\cap S_{\epsilon_t})=\pi_1(X_{t,A})=1$, so necessarily $\pi_1(V)=1$ and the map $\pi_1(X_t\cap S_{\epsilon_t})\to \pi_1(V)$ is an isomorphism. On the other hand, $\pi_1(X_t\cap S_\epsilon)=\pi_1(X_0\cap S_\epsilon)=\pi_1(X_A)=1$, hence the map 
$\pi_1(X_t\cap S_{\epsilon})\to \pi_1(V)$ is also an isomorphism. 

For the homology, we use the Mayer-Vietoris sequence:
$$
\begin{CD}
\dots @>>>  H_{i+1}(X_t\cap B_\epsilon)@>>> H_i(X_t\cap S_{\epsilon_t}) @>>>  H_i(X_t\cap B_{\epsilon_t})\oplus  H_i({V})@>>>\\  @>>>  H_{i}(X_t\cap B_\epsilon) @>>>\dots
\end{CD}
$$
Since $H_{i}(X_t\cap B_\epsilon)=0$ and $H_i(X_t\cap B_{\epsilon_t})=0$, for all $i>0$, we deduce that the map $H_i(X_t\cap S_{\epsilon_t})\to H_i(V)$ is an isomorphism for all $i>0$. Finally, the map $H_i(X_t\cap S_{\epsilon})\to H_i(V)$ is also an isomorphism by the Poincaré duality theorem for cobordism (see \cite{K}).

\begin{figure}[ht!]
\begin{center}
   \psfrag{B}{$B_\epsilon$}
  \psfrag{X0}{$X_0$}
\psfrag{Xt}{$X_t$}
 \psfrag{Bt}{$B_{\epsilon_t}$}
  \psfrag{M}{$M$}
 \psfrag{V}{$V$}
\includegraphics[height=7cm]{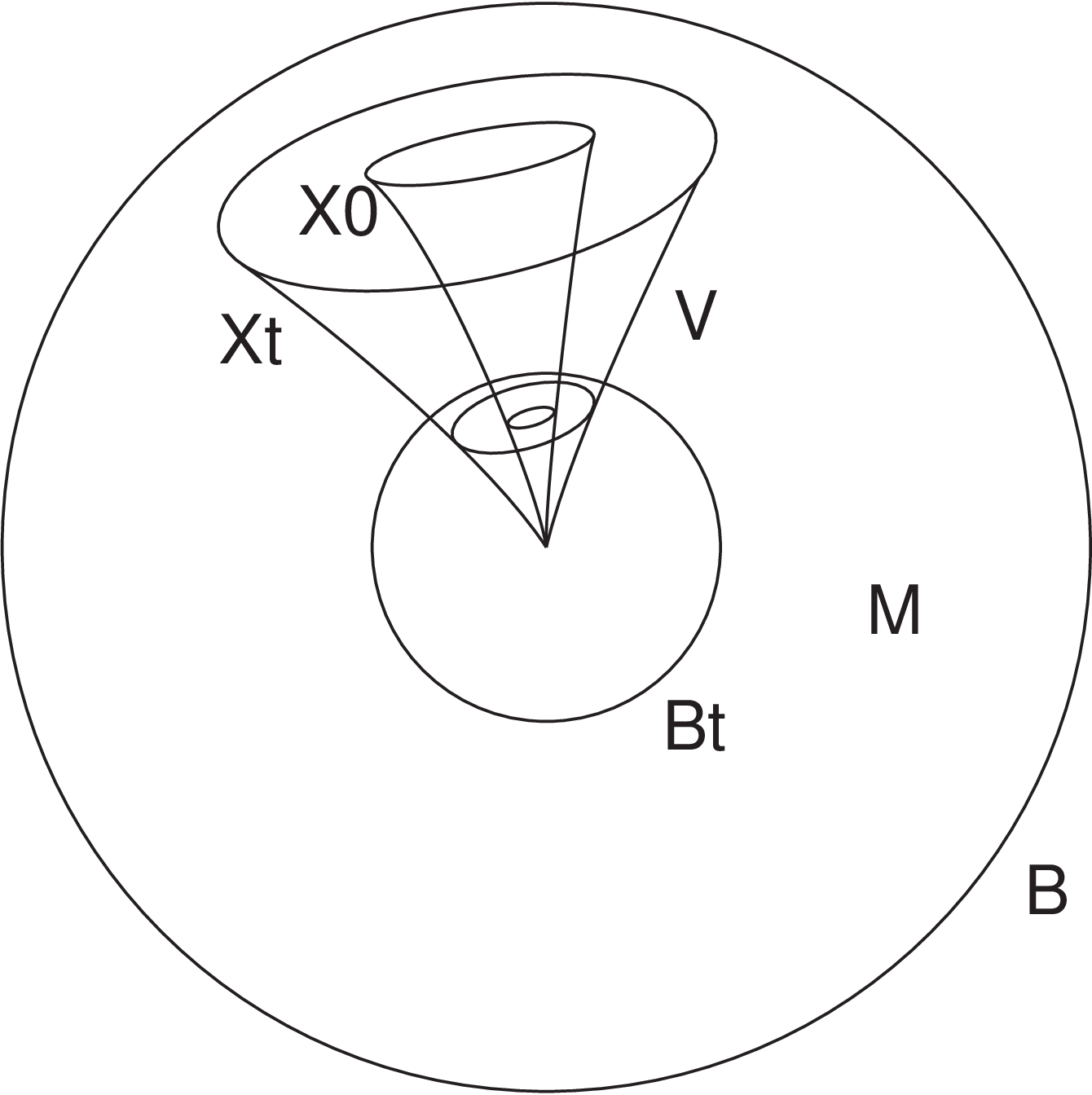}
\caption{}\label{cobordismo}
\end{center}
\end{figure}

Therefore, the pairs $(S_\epsilon, X_t\cap S_\epsilon)$ and $(S_{\epsilon_t}, X_t\cap S_{\epsilon_t})$ are $h$-cobordant (see \cite{S}).
Furthermore, $\dim X_t\cap S_\epsilon=2d-1\geq 5$ and $\codim_{S_\epsilon}(X_t\cap S_\epsilon)=2N-1-(2d-1)=2(N-d)\geq 4$, (where now the dimensions are taken as real manifolds). This implies that $\pi_1(S_\epsilon-S_\epsilon\cap X_t)=1$. Hence, by the relative $h$-cobordism theorem of Smale \cite[Theorem 1.4]{S}, $(S_{\epsilon_t},X_t\cap S_{\epsilon_t})$ is homeomorphic to $(S_\epsilon,X_t\cap S_\epsilon)$.

\end{proof}

We remark that the condition that the determinantal smoothing is $(d-1)$-connected is satisfied in case that $X_t$ is a  curve (\cite{BG}); $d=N-1$, that is, $X_t$ is a hypersurface (\cite{M}) and $s=1$, that is, $X_t$ is an ICIS (\cite{Hamm}). Moreover, it is also known that the constancy of $\nu(X_t,0)$ implies that the family is good in such cases. 
Therefore, Theorem \ref{tipotopconst} indeed extends the results of \cite{LR,Param}.


\section{Whitney Equisingularity}

It is well known that any Whitney equisingular family is topologically trivial by the Thom first isotopy lemma. However, the converse is not true in general. In fact,
the Briançon-Speder example \cite{BS} shows that the $\mu$-constant condition is not a sufficient condition for the Whitney equisingularity of a family (even for a family of hypersurfaces). However, it is known that a family of ICIS $(X_t,0)$ of dimension $d$ is Whitney equisingular if and only if the polar multiplicities $m_i(X_t,0)$, $i=0,\dots,d$ are constant on $t$ (see Gaffney and Teissier \cite{G,T}).
%
In this section, we use our definition of $d$-th polar multiplicity to generalize the results of Gaffney and Teissier in \cite{G,T}.

We start by giving the relationship between the $d$-th polar multiplicity of an IDS of dimension $d$ and of a determinantal deformation of it. 
\begin{lem}
Let $(X_0,0)$ be a $d$-dimensional IDS and let $X_t$ be a determinantal deformation. For a generic linear projection $p:\C^N\to\C$, we have:
\begin{equation*}
m_d(X_0,0)=\sum_{y\in X_t^0} \mu(p|_{X_t},y)+\sum_{x\in S(X_t)}m_d(X_t,x),
\end{equation*}
where $X_t^0$ denotes the smooth part of $X_t$.
\end{lem}

\begin{proof}
We choose a ball $B_\epsilon$ such that the origin is the only singular point of $X_0$ in $B_\epsilon$. Let $\{x_1,\dots,x_{k}\}$ be the singular set of $X_t$ and $\{y_1,\dots,y_{l}\}$ be the singular set of $p|_{X_t^0}$.

For each $i=1,\dots,k$, choose a Milnor disc $D_i$ for $(X_t, x_i)$ and for each $j=1,\dots,l$, choose a Milnor disc $E_j$ around $y_j$ for $p|_{X_t^0}$ (see fig. \ref{deform2}). Let $p_a$ be a generic linear deformation of $p$ and let $X_{t,A}$ be a determinantal smoothing such that ${p_a}|_{X_{t,A}}$ is a Morse function. We have
\begin{eqnarray*}
m_d(X_0,0)&=&\#S(p_a|_{X_{t,A}})\\
&=&\sum_{j=1}^{l} \#S(p_a|_{X_{t,A}\cap E_j})+\sum_{i=1}^{k}\#S(p_a|_{X_{t,A}\cap D_i})\\
&=&\sum_{j=1}^{l} \mu(p|_{X_t},y_j)+\sum_{i=1}^{k}m_d(X_t,x_i).
\end{eqnarray*}
\end{proof}

\begin{figure}[ht!]
\begin{center}
   \psfrag{B}{$B_\epsilon$}
  \psfrag{X0}{$X_0$}
\psfrag{Xt}{$X_t$}
 \psfrag{XtA}{$X_{t,A}$}
  \psfrag{x}{$0$}
 \psfrag{x1}{$x_1$}
\psfrag{x2}{$y_1$}
\psfrag{x3}{$x_2$}
 \psfrag{D1}{$D_1$}
\psfrag{D2}{$E_1$}
\psfrag{D3}{$D_2$}
\psfrag{p}{$p$}
\includegraphics[height=7cm]{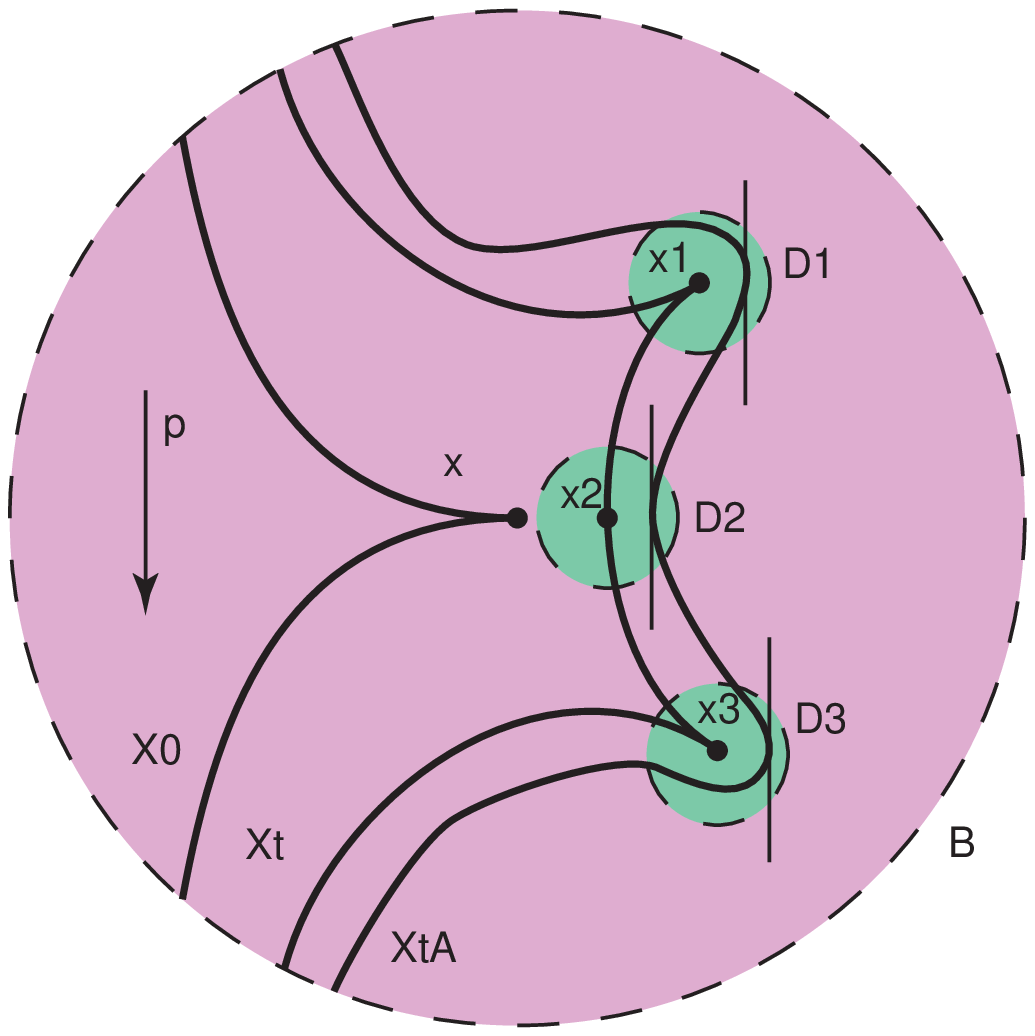}
\caption{}\label{deform2}
\end{center}
\end{figure}

We recall from Teissier \cite{T} the definition of the relative polar multiplicities $m_i(\mathcal X,\pi,0)$, with $i=0,\dots,d$,  of a $(d+1)$-dimensional variety $(\mathcal X,0)$ with respect to $\pi:(\mathcal X,0)\to(\C,0)$. Given a generic linear projection $P:\C^{N+1}\to\C^{d-i+1}$, then $m_i(\mathcal X,\pi,0)=m_0(\overline{S(P|_{\mathcal X^0})},0)$.

\begin{lem}\label{lema3}
Let $\{(X_t,0)\}_{t\in D}$ be a good family of $d$-dimensional IDS. If $m_d(X_t,0)$ is constant on $t\in D$, then $m_d(\mathcal X,\pi,0)=0$.
\end{lem}
\begin{proof}
Choose $\epsilon>0$ such that the single critical point of $X_0\cap B_\epsilon$ is the origin. From the previous lemma, we have
\begin{equation*}
m_d(X_0,0)=\sum_{y\in X_t^0}\mu(p|_{X_t},y)+m_d(X_t,0),
\end{equation*}
where $p:\C^N\to \C$ is any generic linear projection.
Since $m_d(X_t,0)$ is constant and all the terms in the right hand side of the equation are bigger or equal to $0$, $\mu(p|_{X_t},y)=0$ for all $y\in X_t^0$ and hence $p|_{X_t^0}$ is regular. We denote $P(x,t)=(p(x),t)$, then
\begin{equation*}
m_d(\mathcal X,\pi,0)=m_0(\overline{S(P|_{\mathcal X^0})},0)=m_0(\overline{\{(x,t):x\in S(p|_{X_t^0})\}},0)=m(\emptyset,0)=0.
\end{equation*}
\end{proof}

We will present now the main result of this section. The theorem bellow generalizes the results of Gaffney \cite{G} and Teissier \cite{T} by showing necessary and sufficient conditions for a family of IDS to be Whitney equisingular.
\begin{thm}
A good family of $d$-dimensional IDS $\{(X_t,0)\}_{t\in D}$ is Whitney equisingular if and only if all the polar multiplicities $m_i(X_t,0)$, $i=0,\dots,d$ are constant on $t\in D$.
\end{thm}

\begin{proof}
Assume, first, that the family is Whitney-equisingular. From the results of Tessier in \cite{T}, $m_i(\mathcal X,\pi,(0,t))$ is constant on $t\in D$ for $i=1,\dots,d$. In particular, $m_d(\mathcal X,\pi,(0,t))$ is constant on $t\in D$.

We have that $m_d(\mathcal X,\pi,0)=0$. In fact, if it was not true, then $m_d(\mathcal X,\pi,(0,t))=m_d(\mathcal X,\pi,0)\ne0$, for $t\in D$. Therefore, since $(0,t)$ does not belong to $S(P|_{\mathcal X^0})$, we would have 
$$T=\{0\}\times D\subset \overline{S(P|_{\mathcal X^0})}-S(P|_{\mathcal X^0}),$$
but this is not possible because $\overline{S(P|_{\mathcal X^0})}$ is analytic of dimension 1 and hence $\overline{S(P|_{\mathcal X^0})}-S(P|_{\mathcal X^0})$ has dimension 0.

Now if $m_d(\mathcal X,\pi,0)=0$, we have $\overline{S(P|_{\mathcal X^0})}=\emptyset$ and thus, $S(p|_{X_t^0})=\emptyset$. This implies that 
\begin{equation*}
\sum \mu(p|_{X_t^0}, y_i)=0, 
\end{equation*}
and from Lemma \ref{lema3},
$
m_d(X_t,0)=m_d(X_0,0).
$

The fact that the other polar multiplicities are constant follows from Corollary 5.6 of \cite{G}. 

Assume, now, that $m_i(X_t,0)$ is constant for all $i=0,\dots,d$. In particular, $m_d(X_t,0)$ is constant and from Lemma \ref{lema3}, \begin{equation*}
m_d(\mathcal X, \pi,0)=0.
\end{equation*}
Hence, the family $X_t$ is Whitney equisingular by Corollary 5.12 of \cite{G}.
\end{proof}

We remark that in the case of a family of ICIS $\{(X_t,0)\}_{t\in D}$ of dimension $d$, the condition that $m_d(X_t,0)$ is constant on $t$ implies that the family is good (see \cite[5.9]{G}).

\end{document}